\documentclass{amsart}
\usepackage{amssymb}
\pdfpagewidth=\paperwidth
\pdfpageheight=\paperheight

\newtheorem{theorem}{Theorem}[section]
\newtheorem{corollary}{Corollary}[section]
\newtheorem{lemma}{Lemma}[section]
\newtheorem{proposition}{Proposition}[section]
\numberwithin{equation}{section}

\newcommand{\Ai}{\operatorname{Ai}}
\newcommand{\Li}{\operatorname{Li}_2}
\newcommand{\Log}{\operatorname{Log}}
\newcommand{\eps}{\varepsilon}
\newcommand{\C}{\mathbb C}
\newcommand{\dd}{\,\mathrm d}
\newcommand{\Ue}{\mathcal U_\eps}
\DeclareMathOperator{\arctanh}{arctanh}

\title[Airy expansions and the Stieltjes--Wigert kernel]{Uniform Airy expansions for Ramanujan's function and the Stieltjes--Wigert edge kernel}
\author[Y.-T. Li]{Yu-Tian Li}
\address{Nanfang College, Guangzhou, Guangdong 510970, China}
\email{yutianlee@gmail.com}
\date{}

\begin{document}

\begin{abstract}
For $q=e^{-\eps}\to1^-$, we obtain an Airy expansion for Ramanujan's
entire function $A_q(\sqrt q\,x)$, uniformly for $x$ in a sufficiently
small fixed complex neighborhood of $1/4$. After an explicit exponential
normalization, the expansion has analytic coefficients and a nonlinear
Airy argument. We calculate the first corrections and bound the remainder
by $\eps^2$ times an Airy envelope. The estimate is valid at Airy zeros
and retains exponential decay on the real interval below $1/4$.
The proof uses a Gaussian integral representation and an odd analytic
change of variable that reduces the phase to cubic form. The compact
turning-point expansion follows as a corollary. For each fixed positive
zero $\rho_n(q)$, we obtain an expansion of $\log(4\rho_n(q))$ through
order $\eps^2$, with remainder $O_n(\eps^{8/3})$. A zero-free interval
identifies the zeros by their global indices. We also calculate the
$\eps^{2/3}$ correction to the known Airy limit of the Stieltjes--Wigert
edge kernel, with an $O(\eps^{4/3})$ remainder uniformly on compact sets,
including the diagonal.
\end{abstract}
\maketitle

\section{Introduction}
\label{sec:introduction}

For $0<q<1$, write
\[
 (a;q)_k=\prod_{j=0}^{k-1}(1-aq^j),\qquad
 (a;q)_\infty=\prod_{j=0}^{\infty}(1-aq^j),
 \qquad (a;q)_0=1.
\]
Ramanujan's entire function is
\begin{equation}
 A_q(z)=\sum_{k=0}^{\infty}\frac{q^{k^2}(-z)^k}{(q;q)_k}.
 \label{eq:Aq-def}
\end{equation}
An equivalent series occurs in Ramanujan's lost notebook
\cite[p.~57]{Ramanujan1988}, together with a product expansion proved
by Andrews \cite{Andrews2005}. The function also arises in the
asymptotic analysis of $q$-orthogonal polynomials; see Ismail
\cite{Ismail2005,IsmailBook2005} and Ismail and R.~Zhang
\cite{IsmailRZhang2006}.

Several limits of $A_q$ have been studied. At the origin,
R.~Zhang \cite{ZhangRemarks} obtained
$A_q((1-q)z)\to e^{-z}$ as $q\to1^-$, locally uniformly in $z$.
Theta-function asymptotics for large arguments with $q$ fixed, and
coupled limits in which $q\to1^-$ and the argument tends to infinity,
are treated in \cite{ZhangTheta,ZhangScaled}. Li and Wong \cite[Theorem~2]{LiWong2013} obtained the leading
approximation to $A_q(\sqrt q\,x)$ in terms of the standard Airy
function $\Ai$ for $x>1/4$, and discussed its extension to complex
arguments. Haug and Prellberg
\cite{HaugPrellberg2015} derived uniform Airy approximations for the
equivalent series $A_q(x/q)$ and $A_q(x)$ by reducing two coalescing
saddles to cubic form. Their formulas include complex arguments;
see also Haug \cite[Chapter~4]{Haug2017} for a detailed treatment.

With $q=e^{-\eps}$, our expansion is uniform in a fixed complex disk
about $x=1/4$. Its leading term agrees with the Li--Wong expression.
Theorem~\ref{thm:fixed} calculates analytic correction coefficients
and bounds the additive remainder by $\eps^2$ times an Airy envelope.
The estimate retains the exponential decay for real $x<1/4$ and
remains valid at Airy zeros. Under $x=\frac14e^{-\eps^{2/3}\zeta}$,
it gives an explicit correction of order $\eps^{2/3}$ and a further
term $\eps\Ai(\zeta)/48$, uniformly for $\zeta$ in compact subsets
of $\C$.

Hamamoto, Kajiwara, and Witte \cite{HKW2006} obtained Airy
asymptotics for the Kajiwara $q$-Airy function, which is related to
$A_q$ by Morita's connection formula \cite{Morita2011}.
We use a direct integral representation of $A_q$.
The choice of argument $\sqrt q\,x$ puts the reciprocal $q$-product
on a midpoint lattice. After centering at the coalescing saddles, the
phase is odd. An odd analytic cubic transformation then has an even
Jacobian, so the leading coefficient of $\Ai'$ vanishes.
The construction follows the method of Chester, Friedman, and Ursell
\cite{CFU}; see also Wong \cite[Chapter~VII, Section~4]{Wong2001}.
We prove the analytic transformation and the uniform
contour estimates used here, and calculate the corrections by the
standard integration-by-parts procedure; see
\cite{OldeDaalhuisTemme1994}.

The expansion also determines each fixed positive zero as $q\to1^-$.
For fixed $q$, Bergweiler and Hayman \cite{BergweilerHayman} studied
geometric asymptotics of zeros of solutions of linear $q$-difference
equations. Hayman \cite{HaymanReport,Hayman} obtained expansions to
arbitrary order for a $q$-Bessel family containing $A_q(-z)$.
Ismail and C.~Zhang \cite{IsmailCZhang2007} subsequently proved
convergence of the large-zero expansions; see also
\cite{DaiIsmailWang2020,IsmailZeros}. Those results concern the zero
index tending to infinity while $q$ is fixed. Here the index is fixed
and $q\to1^-$. For the $n$th positive zero $\rho_n(q)$ of $A_q$,
Corollary~\ref{cor:zeros} gives the expansion of
$\log(4\rho_n(q))$ through order $\eps^2$, with remainder
$O_n(\eps^{8/3})$. A zero-free interval and the compact Airy limit
identify the locally constructed roots with the first positive zeros.

Baik and Suidan \cite[Section~3]{BaikSuidan2007} obtained Airy and
sine kernel limits for a Stieltjes--Wigert ensemble arising from
nonintersecting Brownian bridges, with the weight varying with
the number of particles.
For fixed $q$, Ramanujan's function occurs in the edge asymptotics
of $q$-orthogonal polynomials \cite{Ismail2005,IsmailLi2013}.
Takahashi \cite{Takahashi2014} and Forrester \cite{Forrester2022}
obtained the subsequent $q\uparrow1$ limit of the Stieltjes--Wigert
edge kernel to the Airy kernel. We calculate the correction of
order $\eps^{2/3}$ under $u=\frac14e^{-\eps^{2/3}s}$, $s\in\mathbb R$,
with the transformed kernel taken relative to Lebesgue measure in $s$.
After analytic continuation, Proposition~\ref{prop:kernel} gives an
$O(\eps^{4/3})$ remainder uniformly on compact subsets of $\C^2$,
including the diagonal. The limit in matrix size is taken at fixed
$q$ before $q\uparrow1$.

The main results are stated in Section~\ref{sec:results}.
Section~\ref{sec:main-proof} proves the uniform expansion.
Section~\ref{sec:consequences} derives the compact-scale and zero
formulas and compares the normalized difference equation with the Airy
equation. Section~\ref{sec:kernel} applies the compact expansion to the
Stieltjes--Wigert edge kernel. Numerical illustrations are given in
Section~\ref{sec:numerics}.

\section{Main Results}
\label{sec:results}
Put $q=e^{-\eps}$, where $\eps>0$. All powers of $\eps$ are positive.
Near $x=1/4$ use the logarithm real on the positive real axis, and write
\begin{equation}
 C(x)=\frac{(\Log x)^2}{4}-\frac{\pi^2}{12},\qquad
 T_\eps(x)=\frac{\eps^{1/6}}{2\sqrt\pi}
 e^{-C(x)/\eps}A_q(\sqrt q\,x).
 \label{eq:fixed-normalization}
\end{equation}

We use the standard Airy function, normalized by
\begin{equation}
 \Ai(w)=\frac1{2\pi i}\int_{\Gamma_{\rm Ai}}
 e^{v^3/3-wv}\,\dd v,
 \label{eq:airy-contour}
\end{equation}
where the contour runs from infinity with argument $-\pi/3$ to
infinity with argument $\pi/3$.
We first define an analytic turning-point coordinate $\chi(x)$.
For real $x<1/4$ sufficiently close to $1/4$, let $\chi(x)>0$ be
determined by
\begin{equation}
 \frac23\chi(x)^{3/2}
 =\int_{\log(4x)}^0\arctanh\sqrt{1-e^v}\,\dd v.
 \label{eq:chi-real}
\end{equation}
This function has a holomorphic continuation to a neighborhood of $1/4$,
as proved in Section~\ref{sec:normalform}, with
\begin{equation}
 \chi(1/4)=0,\qquad \chi'(1/4)=-4.
 \label{eq:chi-initial}
\end{equation}
Define the analytic coefficients
\begin{equation}
 a(x)=\left(\frac{\chi(x)}{1-4x}\right)^{1/4},\qquad
 b(x)=\frac{(1+16x)a(x)^6-5}{48a(x)^7(1-4x)^2},
 \label{eq:coefficients}
\end{equation}
using the fourth root with $a(1/4)=1$. The apparent singularity in
$b$ is also removable, and
\begin{equation}
 b(1/4)=\frac1{560}.
 \label{eq:b-origin}
\end{equation}

For $w\in\C$, set
\begin{equation}
 E(w)=(1+|w|)^{-1/4}
 \exp\left(-\frac23\operatorname{Re}w^{3/2}\right).
 \label{eq:envelope}
\end{equation}
The power is principal off the negative real axis; its real part is
extended continuously to that axis. Thus $E$ is positive and continuous.
For real $w$,
\[
 E(w)=(1+|w|)^{-1/4}
 \exp\left(-\frac23(\max\{w,0\})^{3/2}\right).
\]

\begin{theorem}[Uniform Airy expansion]\label{thm:fixed}
Let $q=e^{-\eps}$, and let $\chi$ be the analytic continuation
determined by \eqref{eq:chi-real}--\eqref{eq:chi-initial}.
Let $a,b$ be given by \eqref{eq:coefficients} and $E$ by
\eqref{eq:envelope}, with the branches specified above.
There are constants $\delta>0$, $\eps_0>0$ and $M>0$ such that,
uniformly for $x\in\C$ with $|x-1/4|\le\delta$ and
$0<\eps\le\eps_0$,
\begin{align}
 &\frac{\eps^{1/6}}{2\sqrt\pi}
 \exp\left\{\frac{\pi^2}{12\eps}
             -\frac{(\Log x)^2}{4\eps}\right\}
 A_q(\sqrt q\,x)\notag\\
 &\qquad=a(x)\left(1+\frac{\eps}{48}\right)
       \Ai\left(\frac{\chi(x)}{\eps^{2/3}}\right)
   +\eps^{4/3}b(x)
       \Ai'\left(\frac{\chi(x)}{\eps^{2/3}}\right)
   +R_\eps(x),\label{eq:fixed-expansion}\\
 &|R_\eps(x)|\le M\eps^2
 E\left(\frac{\chi(x)}{\eps^{2/3}}\right).
 \label{eq:fixed-error}
\end{align}
In particular, with $T_\eps$ defined by \eqref{eq:fixed-normalization},
\begin{equation}
 T_\eps(x)=a(x)\Ai\left(\frac{\chi(x)}{\eps^{2/3}}\right)
 +O\left(\eps E\left(\frac{\chi(x)}{\eps^{2/3}}\right)\right)
 \label{eq:leading}
\end{equation}
on the same fixed disk.
\end{theorem}

The estimate is additive and remains meaningful at Airy zeros. It also
retains the exponential decay on the real interval $x<1/4$.
The radius $\delta$ is independent of $\eps$. In the original
variable $z=\sqrt q\,x$, the domain contains
$|z-1/4|\le\delta/2$ for all sufficiently small $\eps$.

For real $x>1/4$ close to $1/4$, the analytic continuation in
\eqref{eq:chi-real} satisfies
\begin{equation}
 \frac23(-\chi(x))^{3/2}
 =\int_0^{\log(4x)}\arctan\sqrt{e^v-1}\,\dd v.
 \label{eq:chi-oscillatory}
\end{equation}
Consequently the leading term is precisely the Airy expression of
Li and Wong \cite[Theorem~2]{LiWong2013}, with their
$\xi=-\chi/\eps^{2/3}$. Theorem~\ref{thm:fixed} gives an additive error bound that is uniform
as $x$ passes through $1/4$.

For the compact turning-point scale, put
\begin{equation}
 z_\eps(\zeta)=\frac14e^{-\eps/2-\eps^{2/3}\zeta}
 \label{eq:turning-scaling}
\end{equation}
and define the entire function
\begin{equation}
\begin{split}
 \Ue(\zeta)={}&\frac{\eps^{1/6}}{2\sqrt\pi}
 \exp\left\{\frac{\pi^2-3(\log4)^2}{12\eps}
 -\frac{(\log4)\zeta}{2\eps^{1/3}}
 -\frac{\eps^{1/3}\zeta^2}{4}\right\}
 A_q(z_\eps(\zeta)).
\end{split}
\label{eq:normalization}
\end{equation}
For $\zeta$ in a fixed compact subset of $\C$ and sufficiently small
$\eps$, the point $x=\frac14e^{-\eps^{2/3}\zeta}$ lies in the chosen
logarithm neighborhood, where $\Log x=-\log4-\eps^{2/3}\zeta$.
Since $\sqrt q\,x=z_\eps(\zeta)$, substitution in
\eqref{eq:fixed-normalization} gives the exact identity
\[
 \Ue(\zeta)=T_\eps\left(\frac14e^{-\eps^{2/3}\zeta}\right).
\]

\begin{corollary}[Compact turning-point scale]
\label{prop:main-expansion}
For each compact set $K\subset\C$, uniformly for $\zeta\in K$,
\begin{align}
 \Ue(\zeta)
 ={}&\Ai(\zeta)+\frac{\eps^{2/3}}{30}
 \bigl(4\zeta\Ai(\zeta)+\zeta^2\Ai'(\zeta)\bigr)\notag\\
 &+\frac\eps{48}\Ai(\zeta)+O_K(\eps^{4/3})
 \qquad(\eps\downarrow0).
 \label{eq:main-expansion}
\end{align}
\end{corollary}

Earlier Airy approximations for $A_q$ are given in
\cite{HaugPrellberg2015,Haug2017}. The compact leading limit
$\Ue\to\Ai$ also follows by combining the estimates in
\cite[Proposition~3.8]{HKW2006} with Morita's connection formula
\cite{Morita2011}. Corollary~\ref{prop:main-expansion} gives the
explicit corrections and their uniform remainder.

\begin{corollary}[Fixed positive zeros]
\label{cor:zeros}
Let $a_1>a_2>\cdots$ be the negative zeros of $\Ai$, and let
$\rho_n(q)$ denote the $n$th positive zero of $A_q$.
For each fixed $n$, as $\eps\downarrow0$,
\begin{align}
 \log\bigl(4\rho_n(e^{-\eps})\bigr)
 ={}&-a_n\eps^{2/3}-\frac\eps2
       +\frac{a_n^2}{30}\eps^{4/3}\notag\\
 &+\left(\frac1{560}-\frac{a_n^3}{525}\right)\eps^2
       +O_n(\eps^{8/3}).
 \label{eq:zero-exp-form}
\end{align}
The remainder is uniform over a fixed finite set of indices.
In particular,
\begin{equation}
 \rho_n(e^{-\eps})\longrightarrow\frac14,
 \qquad
 \frac{\rho_{n+1}(e^{-\eps})-\rho_n(e^{-\eps})}{\eps^{2/3}}
 \longrightarrow\frac{a_n-a_{n+1}}4.
 \label{eq:zero-limits}
\end{equation}
\end{corollary}

The existence of an infinite sequence of positive zeros of $A_q$ is
established in \cite{IsmailCZhang2007}.

\section{Proof of the Uniform Expansion}
\label{sec:main-proof}

\subsection{The integral representation and centered phase}

For $|z|<1$, the $q$-binomial theorem and the Gaussian Fourier transform give
\begin{equation}
 A_q(z)=\frac1{\sqrt{4\pi\eps}}
 \int_{\mathbb R}
 \frac{e^{-t^2/(4\eps)}}{(-ze^{it};e^{-\eps})_\infty}\,\dd t.
 \label{eq:gaussian}
\end{equation}
Indeed, expansion of the reciprocal product gives $e^{ikt}$, whose
normalized Gaussian integral is $e^{-\eps k^2}$. For $|z|\le r<1$,
the expansion is dominated by $((q;q)_\infty(1-r))^{-1}$, which
justifies termwise integration.

The argument $z=e^{-\eps/2}x$ places the product at the midpoints of
the exponential lattice. This removes the order-one amplitude from its
Euler--Maclaurin expansion.

Here $\Li$ denotes the dilogarithm on $\C\setminus[1,\infty)$,
with the branch determined by
\[
 \Li(w)=-\int_0^w\frac{\Log(1-s)}s\,\dd s.
\]

\begin{lemma}\label{lem:midpoint}
Uniformly for $w$ in a compact subset of $\C\setminus[1,\infty)$,
\begin{equation}
 \log\frac1{(we^{-\eps/2};e^{-\eps})_\infty}
 =\frac{\Li(w)}{\eps}-\frac{\eps w}{24(1-w)}+O(\eps^3).
 \label{eq:midpoint}
\end{equation}
The logarithm is the analytic sum of the principal logarithms of the
factors, and the remainder is locally holomorphic in $w$.
\end{lemma}

\begin{proof}
Apply midpoint Euler--Maclaurin summation to
$f_w(v)=-\Log(1-we^{-v})$. Since
$f'_w(0)=-w/(1-w)$, its first correction is
$\eps f'_w(0)/24$. Also
\[
 \int_0^\infty f_w(v)\,\dd v=\Li(w).
\]
On a slightly larger cut-free compact set, the radial contractions
$we^{-v}$ stay a positive distance from $1$. The derivatives of $f_w$
up to order four have uniformly bounded integrals on $[0,\infty)$;
they decay exponentially as $v\to\infty$. The next endpoint term and
the integral remainder are $O(\eps^3)$. Holomorphy follows from local
uniform convergence of the logarithmic sums.
\end{proof}

With $w=-xe^{it}$ the phase is
\begin{equation}
 \phi(t;x)=-\frac{t^2}{4}+\Li(-xe^{it}).
 \label{eq:phase}
\end{equation}
Put $\tau=\Log(4x)$ and make the exact substitution
\begin{equation}
 t=i\Log x+2iu.
 \label{eq:center}
\end{equation}
Then, near $u=0$,
\begin{equation}
 \phi(t;x)=C(x)+F(u;\tau),\qquad
 F(u;\tau)=\tau u+2\int_0^u\Log\cosh v\,\dd v.
 \label{eq:odd-phase}
\end{equation}
To verify the identity, differentiate both sides with respect to $u$
and check their values at $u=0$. The relation
$\Log(1+e^{-2u})=\log2-u+\Log\cosh u$ holds on this neighborhood.
In particular $F$ is odd, and
\begin{align}
 F_u(u;\tau)&=\tau+2\Log\cosh u,\label{eq:phase-derivative}\\
 F(u;0)&=\frac{u^3}{3}-\frac{u^5}{30}
       +\frac{2u^7}{315}-\frac{17u^9}{11340}+O(u^{11}).
 \label{eq:phase-series}
\end{align}
The product amplitude from \eqref{eq:midpoint}, in this coordinate, is
\begin{equation}
 1+\frac{\eps}{48}(1-\tanh u)+O(\eps^2),
 \label{eq:midpoint-amplitude}
\end{equation}
uniformly and analytically on a fixed small neighborhood of $u=0$.

\subsection{An odd analytic cubic transformation}\label{sec:normalform}

The function $-2\Log\cosh\sqrt r$ is analytic at $r=0$ and has
derivative $-1$ there. Let $r=r(\tau)$ denote its analytic inverse:
\begin{equation}
 \tau=-2\Log\cosh\sqrt r.
 \label{eq:r-coordinate}
\end{equation}
The two critical points of $F$ are $u=\pm\sqrt r$. Define
\begin{equation}
 c(r)=3\int_0^1 y^2\frac{\tanh(\sqrt r\,y)}{\sqrt r\,y}\,\dd y,
 \qquad c(0)=1,
 \qquad \chi=r\,c(r)^{2/3},
 \label{eq:chi-analytic}
\end{equation}
with powers of $c$ chosen near $1$. These formulas are single-valued
and holomorphic in $r$. If $s^2=r$, integration by parts gives
\[
 F(s;\tau)=-2\int_0^s v\tanh v\,\dd v
           =-\frac23c(r)s^3.
\]
For real $\tau<0$, substitution $\tau=-2\log\cosh s$ gives
\eqref{eq:chi-real}; for $\tau>0$, put
$s=i\arctan\sqrt{e^\tau-1}$ to obtain \eqref{eq:chi-oscillatory}.
This proves the asserted analytic continuation of $\chi$.

\begin{lemma}\label{lem:normalform}
There is an odd function $u=V(v,r)$, jointly holomorphic on a fixed
neighborhood of $(0,0)$, such that
\begin{equation}
 F(V(v,r);\tau(r))=\frac{v^3}{3}-\chi(r)v,
 \qquad V_v(0,0)=1.
 \label{eq:normalform}
\end{equation}
After shrinking the two disks, the map is biholomorphic in $v$ uniformly
in $r$, and maps $\pm\sqrt\chi$ to $\pm\sqrt r$ with consistent signs.
\end{lemma}

\begin{proof}
Write $P_r(w)=w^3/3-rw$. The odd function
$F(w;\tau(r))-c(r)P_r(w)$ has double zeros at $w=\pm\sqrt r$,
because the two functions have the same critical points and critical
values. Analytic division therefore gives
\begin{equation}
 F(w;\tau(r))-c(r)P_r(w)=(w^2-r)^2 H(w,r),
 \label{eq:double-division}
\end{equation}
with $H$ analytic and odd; in particular, $H(0,0)=0$.
To include $r=0$ explicitly, division by
the monic polynomial $(w^2-r)^2$ leaves an odd remainder
$h_1(r)w+h_3(r)w^3$. Its value and derivative at $\sqrt r$, for
$r\ne0$, imply $h_1+rh_3=h_1+3rh_3=0$. Both coefficients vanish,
and analyticity extends this conclusion to $r=0$.

Similarly $F_u(w;\tau(r))=(w^2-r)D(w,r)$ with $D(0,0)=1$.
Seek $U(w,r)=w+(w^2-r)h(w,r)$. After division by $(w^2-r)^2$,
the equation $F(U;\tau(r))=c(r)P_r(w)$ becomes
\[
 H+D h+\sum_{k\ge2}\frac{F^{(k)}(w;\tau(r))}{k!}
             (w^2-r)^{k-2}h^k=0.
\]
This is an analytic equation near $(w,r,h)=(0,0,0)$, with derivative
$1$ in $h$. The implicit function theorem supplies $h$; uniqueness
and oddness give $h(-w,r)=-h(w,r)$. Thus $U$ is odd and
$U_w(0,0)=1$. Set $V(v,r)=U(v/c(r)^{1/3},r)$.
Equation \eqref{eq:normalform} follows from
$cP_r(v/c^{1/3})=v^3/3-r c^{2/3}v$.
On smaller fixed disks, $|V_v-1|<1/2$. Integration along straight
segments proves injectivity in $v$, and the derivative never vanishes.
\end{proof}

Below, $r$ is determined by $x$ through \eqref{eq:r-coordinate},
with $\tau=\Log(4x)$. Write
\begin{equation}
 g(v,x)=V_v(v,r),\qquad P(v;x)=\frac{v^3}{3}-\chi(x)v.
 \label{eq:jacobian}
\end{equation}
The Jacobian $g$ is even. Differentiating \eqref{eq:normalform} twice
at $v=\sqrt\chi$, where $V=\sqrt r=s$, gives
\[
 g(\sqrt\chi,x)^2=\frac{\sqrt\chi}{\tanh s}.
\]
Since $\tanh^2s=1-4x$ and $g(0,1/4)=1$,
\begin{equation}
 g(\sqrt\chi,x)=g(-\sqrt\chi,x)=a(x).
 \label{eq:leading-coefficient}
\end{equation}
The leading odd coefficient, which would multiply $\Ai'$, is therefore
zero for every $x$ in the neighborhood.

\subsection{Uniform bounds on the Airy contour}

\begin{lemma}\label{lem:airy-bound}
For $\eps>0$ and $\chi\in\C$, there is a contour in the Airy
homology class, oriented from the lower to the upper decay sector, for
which
\begin{equation}
 \int_\Gamma\left|e^{(v^3/3-\chi v)/\eps}\right|\,|\dd v|
 \le K\eps^{1/3}E(\chi/\eps^{2/3}),
 \label{eq:l1-bound}
\end{equation}
with an absolute constant $K$. With an additional factor $|v|$ in the
integrand, the bound is multiplied by at most
$K(\sqrt{|\chi|}+\eps^{1/3})$. The estimate is uniform in the
argument of $\chi$, including the negative real axis.
\end{lemma}

\begin{proof}
First let $\chi\ne0$, $\theta=\arg\chi$, and $|\theta|\le5\pi/6$.
With $s=\sqrt\chi$ principal, take the rays
\[
 v=s+t e^{i\beta_\pm},\qquad
 \beta_\pm=\pm\frac\pi3-\frac\theta6,\qquad t\ge0.
\]
The lower ray is traversed toward $s$. On either ray,
\[
 P(v)= -\frac23\chi^{3/2}
       +s t^2e^{2i\beta_\pm}+\frac{t^3}{3}e^{3i\beta_\pm}.
\]
The angles of the last coefficients are
$\pm2\pi/3+\theta/6$ and $\pm\pi-\theta/2$.
Their cosines are bounded above by negative constants. Thus
\begin{equation}
 \operatorname{Re}(P(v)-P(s))
 \le-c\bigl(\sqrt{|\chi|}\,t^2+t^3\bigr).
 \label{eq:sector-decay}
\end{equation}
The two directions stay inside the standard Airy decay sectors.

Near the negative axis write $\chi=-d^2e^{i\varphi}$, $d\ge0$,
$|\varphi|\le\pi/6$. For $\chi=-d^2$, use the outer rays
$v=\pm id+t e^{\pm i\pi/3}$ and the inner segments
$v=-t\pm i(d-t)$, $0\le t\le d$.
They connect the lower ray to the upper through $-id,-d,id$.
On the upper outer ray and inner segment respectively,
\begin{align*}
 P(v)-P(id)&=-\frac{\sqrt3}{2}dt^2-\frac{t^3}{3}
                         -\frac{i}{2}dt^2,\\
 P(v)-P(id)&=-\frac23t^2(3d-t)-\frac{2i}{3}t^3.
\end{align*}
The ratios of the absolute imaginary part to the negative real part
are at most $1/\sqrt3$ and $1/2$. The lower pieces are conjugate.
Rotate the contour by $\varphi/2$. The phase is multiplied by
$e^{3i\varphi/2}$; since $|3\varphi/2|\le\pi/4$, both strict
decay bounds persist with absolute constants. The larger of the two
factors $|e^{P(\pm\sqrt\chi)/\eps}|$ is
\[
 \exp\left(-\frac{2}{3\eps}
                 \operatorname{Re}\chi^{3/2}\right),
\]
with the principal power as in \eqref{eq:envelope}.

It remains to integrate the bounds. Put $d=\sqrt{|\chi|}$.
The outer integrals are bounded by
\[
 \int_0^\infty e^{-c(dt^2+t^3)/\eps}\,\dd t
 \le K\eps^{1/3}(1+d^2/\eps^{2/3})^{-1/4}.
\]
For $d\le\eps^{1/3}$ use the cubic term, and for
$d\ge\eps^{1/3}$ use the Gaussian term. The inner segments are
bounded by $K\min\{d,\sqrt{\eps/d}\}$, which obeys the same bound
in those two ranges. A factor $|v|\le d+t$ gives the stated moment
estimate. When $\chi=0$, the usual two Airy rays give the result
directly. All paths have the Airy endpoints and, since the cubic
integrand is entire, belong to the required homology class.
\end{proof}

It follows in particular that
\begin{equation}
 |\Ai(w)|\le K E(w),\qquad
 |\Ai'(w)|\le K(1+|w|)^{1/2}E(w).
 \label{eq:airy-bounds}
\end{equation}

\subsection{Localization of the original contour}

We give the global estimate needed to apply the local transformation.
Let $L=\log4$, $t_0=-iL$, and $t_* =\sqrt3L<\pi$.
Deform the central interval $[-\pi,\pi]$ to the outer real segments
$[-\pi,-t_*]$, $[t_*,\pi]$ and the two straight segments joining
$-t_*$ to $t_0$ to $t_*$. The reciprocal product has poles at
\[
 t=(2m+1)\pi-\arg x
   +i\bigl(\log|x|-\eps(j+1/2)\bigr),
 \qquad m\in\mathbb Z,\quad j\ge0.
\]
For $x$ close to $1/4$, they remain a positive distance from the
closed deformation triangle. The image $-xe^{it}$ also stays in a
fixed compact subset of $\C\setminus[1,\infty)$ on a neighborhood
of the deformed contour. Thus the deformation and
Lemma~\ref{lem:midpoint} apply uniformly in $x$ and small $\eps$.

On the two sloping pieces, \eqref{eq:center} at $x=1/4$ gives
$u=t e^{\pm i\pi/3}$, $0\le t\le L$. On the upper ray write
$u=\alpha+i\sqrt3\alpha$, $0\le\alpha\le\log2$.
Since $|\operatorname{Im}u|<\pi/2$ on these segments, the phase
identity \eqref{eq:odd-phase} holds throughout them.
For $0<\alpha\le\log2$,
\[
 \sinh\alpha<\sin(\sqrt3\alpha).
\]
Indeed $\sin(\sqrt3\alpha)-\sinh\alpha$ is strictly concave,
vanishes at zero, and is
positive at $\log2$: use
$\sin(\sqrt3\log2)>\sin(2/\sqrt3)
\ge14/(9\sqrt3)>3/4=\sinh\log2$.
Consequently $|\cosh u|^2=\cos^2(\sqrt3\alpha)+\sinh^2\alpha<1$,
and $\Log\cosh u$ has negative real part and positive imaginary part.
Thus
$\operatorname{Re}(e^{i\pi/3}\Log\cosh u)<0$.
Together with \eqref{eq:phase-series}, integration gives
\[
 \operatorname{Re}F(t e^{\pm i\pi/3};0)\le-c t^3
 \quad(0\le t\le L)
\]
for some $c>0$. On the outer positive real segment,
\[
 \frac{\dd}{\dd t}\operatorname{Re}\phi(t;1/4)
 =-\frac t2+\arg(1+e^{it}/4)<0,
\]
because $\arg(1+e^{it}/4)\le\arcsin(1/4)<t_*/2$.
The negative side is its conjugate. There is therefore a fixed negative
phase gap outside every fixed small neighborhood of $t_0$.
Continuity preserves a smaller gap for complex $x$ near $1/4$,
relative to $C(x)$.

Choose $R>0$ so that the map in Lemma~\ref{lem:normalform} is
biholomorphic on a neighborhood of $|v|\le3R$, uniformly for the
parameters in a fixed small disk. At $x=1/4$, the transformed sloping
pieces meet $|v|=R$ transversally at arguments close to $\pm\pi/3$.
After decreasing $R$ and then the parameter disk, the endpoints remain
in fixed compact subsectors of the two Airy decay sectors. Arrange also
$|\sqrt\chi|<R/100$.

Truncate the contours of Lemma~\ref{lem:airy-bound} at $|v|=R$.
Their endpoints lie in compact subsectors of the same decay sectors.
Join corresponding endpoints by arcs on $|v|=R$ contained in those
subsectors. On the arcs, $\operatorname{Re}v^3\le-cR^3$ for some
$c>0$. Decreasing the parameter disk so that $|\chi|R<cR^3/6$
gives $\operatorname{Re}P\le-cR^3/6$ there. The replacement stays
within the common holomorphic chart, and Cauchy's theorem applies.
The arcs and the removed portions of the fixed contour have a uniform
negative phase gap. All disks, radii and gaps are chosen before
$\eps$ tends to zero.

The portion of \eqref{eq:gaussian} over $|t|\ge\pi$ is also
exponentially small. For $|\sqrt q\,x|\le1/3$,
\[
 \left|\frac1{(-\sqrt q\,xe^{it};q)_\infty}\right|
 \le\frac1{(1/3;q)_\infty}
 \le\frac32\exp\left(\frac{\Li(1/3)}\eps\right),
\]
and the Gaussian tail is at most
$4\eps\pi^{-1}e^{-\pi^2/(4\eps)}$.
At $x=1/4$,
$\pi^2/4-\Li(1/3)+C(1/4)>0$; continuity preserves this gap.

It follows that all omitted pieces in the normalized integral are
$O(\eps^{-1/3}e^{-\gamma/\eps})$ for some fixed $\gamma>0$.
Choose the disk smaller once more so that
$\frac23\sup|\chi|^{3/2}<\gamma/2$.
Then these errors are $O(\eps^m E(\chi/\eps^{2/3}))$ for every
fixed $m$. This comparison is needed on the exponentially decaying side.

\subsection{The coefficients and the remainder}

Set $X=\chi(x)/\eps^{2/3}$. On the truncated Airy contour write
\[
 \mathcal J_\eps[h]=\frac{\eps^{-1/3}}{2\pi i}
       \int e^{P(v;x)/\eps}h(v)\,\dd v.
\]
The original $t$ contour induces the reverse Airy orientation under
$t=i\Log x+2iu$. Since $\dd t=2i\,\dd u$, the normalization
in \eqref{eq:fixed-normalization} gives the positive sign
\begin{equation}
 \begin{split}
 T_\eps(x)&=\mathcal J_\eps\left[
 g(v,x)\left\{1+\frac\eps{48}
                (1-\tanh V(v,r))+O(\eps^2)\right\}\right]\\
 &\quad+\text{exponentially small terms}.
 \end{split}
 \label{eq:canonical-integral}
\end{equation}
The analytic $O(\eps^2)$ term is uniform on a fixed local region.
By Lemma~\ref{lem:airy-bound}, a uniformly bounded local amplitude
contributes $O(E(X))$. Extending the constant and linear amplitudes to
the complete Airy contour gives
\begin{equation}
 \mathcal J_\eps[1]=\Ai(X),\qquad
 \mathcal J_\eps[v]=-\eps^{1/3}\Ai'(X),
 \label{eq:airy-moments}
\end{equation}
up to exponentially small errors. Endpoint terms in the following
integrations by parts have the same bound and are absorbed throughout.
At the corners of a piecewise contour the boundary terms cancel.

Because $g$ is even and takes the value $a$ at both saddles, divide
analytically to obtain
\[
 g=a+(v^2-\chi)h_0,\qquad h_0\ \text{even}.
\]
Uniform boundedness of this quotient and its derivatives follows on
nested fixed disks. For example, write $g(v,x)=G(v^2,x)$; then $h_0$
is the divided difference of $G$ at $v^2$ and $\chi$.
Since $P'=v^2-\chi$, integration by parts gives
\[
 \mathcal J_\eps[g]=a\Ai(X)-\eps\mathcal J_\eps[h'_0].
\]
The derivative is odd. A second analytic division gives
\[
 h'_0=\beta v+(v^2-\chi)h_1,\qquad
 \beta=\frac{h'_0(\sqrt\chi,x)}{\sqrt\chi},
\]
with removable values at $\chi=0$. Hence
\begin{equation}
 \mathcal J_\eps[g]=a\Ai(X)
    +\eps^{4/3}\beta\Ai'(X)+O(\eps^2E(X)).
 \label{eq:jacobian-expansion}
\end{equation}
Likewise $g\tanh V$ is odd, and its linear coefficient in division by
$v^2-\chi$ is
\[
 \frac{g(\sqrt\chi,x)\tanh\sqrt r}{\sqrt\chi}=\frac1a.
\]
One integration by parts therefore yields
\begin{equation}
 \mathcal J_\eps[g\tanh V]
 =-\frac{\eps^{1/3}}a\Ai'(X)+O(\eps E(X)).
 \label{eq:odd-amplitude}
\end{equation}
In the $\eps\mathcal J_\eps[g]/48$ contribution of
\eqref{eq:canonical-integral}, one integration by parts suffices to
replace $\mathcal J_\eps[g]$ by $a\Ai(X)$ at a cost
$O(\eps^2E(X))$. Thus
\begin{equation}
 T_\eps(x)=a\left(1+\frac\eps{48}\right)\Ai(X)
  +\eps^{4/3}\left(\beta+\frac1{48a}\right)\Ai'(X)
  +O(\eps^2E(X)).
 \label{eq:before-coefficient}
\end{equation}

To evaluate the remaining coefficient, put $d=\tanh\sqrt r$,
$p=4x=1-d^2$ and $v_0=\sqrt\chi=a^2d$. Differentiation of
$F(V(v))=P(v)$ at $v_0$ gives
\begin{align*}
 g'(v_0)&=\frac{1-pa^3}{3da},\\
 g''(v_0)&=\frac{pa^3}{2}-\frac{p(1-pa^3)}{2d^2}
                       -\frac{(1-pa^3)^2}{12d^2a^3}.
\end{align*}
Expansion of $(g-a)/(v^2-\chi)$ at $v_0$ gives
\[
 \beta=\frac{g''(v_0)}{4v_0^2}-\frac{g'(v_0)}{4v_0^3}.
\]
Substitution simplifies to
\[
 \beta+\frac1{48a}
 =\frac{(1+16x)a^6-5}{48a^7(1-4x)^2}=b(x).
\]
Analyticity at the turning point follows already from the division
construction. At $\tau=0$, inversion of \eqref{eq:phase-series} gives
\[
 V(v,0)=v+\frac{v^3}{30}-\frac{v^5}{525}+O(v^7),\qquad
 g(v,1/4)=1+\frac{v^2}{10}-\frac{v^4}{105}+O(v^6).
\]
Thus $\beta(1/4)=-2/105$, and
$b(1/4)=-2/105+1/48=1/560$.
This proves \eqref{eq:fixed-expansion}--\eqref{eq:fixed-error}.
Finally \eqref{eq:airy-bounds} and
\[
 \eps^{1/3}(1+|X|)^{1/2}
 =\sqrt{\eps^{2/3}+|\chi|}
\]
give \eqref{eq:leading} on the fixed disk, completing the proof of
Theorem~\ref{thm:fixed}.

\section{The Compact Scale and Positive Zeros}
\label{sec:consequences}

\subsection{Proof of Corollary~\ref{prop:main-expansion}}

At $\tau=\Log(4x)=0$, expansion of the analytic definitions gives
\begin{equation}
 \chi=-\tau+\frac{\tau^2}{30}-\frac{\tau^3}{3150}+O(\tau^4),
 \qquad
 a=1-\frac{2\tau}{15}-\frac{\tau^2}{630}+O(\tau^3).
 \label{eq:parameter-series}
\end{equation}
Put $\eta=\eps^{1/3}$ and $x=\frac14e^{-\eta^2\zeta}$.
Uniformly for $\zeta$ in a fixed compact set $K\subset\C$,
\[
 \frac{\chi(x)}{\eta^2}
 =\zeta+\frac{\eta^2\zeta^2}{30}
       +\frac{\eta^4\zeta^3}{3150}+O_K(\eta^6),
 \qquad
 a(x)=1+\frac{2\eta^2\zeta}{15}+O_K(\eta^4).
\]
Taylor expansion of $\Ai$ in Theorem~\ref{thm:fixed} proves
Corollary~\ref{prop:main-expansion}. The term $\eps\Ai/48$ is
multiplicative and therefore causes no zero displacement at that order.

\subsection{Proof of Corollary~\ref{cor:zeros}}
\label{sec:zero-proof}

\begin{proof}
Write $\eta=\eps^{1/3}$ and use $X=\chi(x)/\eta^2$ as the local variable, so
$x=\chi^{-1}(\eta^2X)$. On any fixed disk about a simple Airy zero,
divide \eqref{eq:fixed-expansion} by the nonzero factor
$a(x)(1+\eta^3/48)$. Since $b/a=1/560+O_n(\eta^2)$ there, the
result is
\[
 \Ai(X)+\frac{\eta^4}{560}\Ai'(X)+O_n(\eta^6).
\]
Rouch\'e's theorem first gives one zero in a small fixed disk; simplicity
of the Airy zero then gives
\[
 X_{n,\eps}=a_n-\frac{\eta^4}{560}+O_n(\eta^6).
\]
Conjugation symmetry makes this zero real. Write the corresponding
argument as $x=\frac14e^{-\eta^2\zeta_{n,\eps}}$. Inverting
\eqref{eq:parameter-series} in the compact coordinate gives
\begin{equation}
 \zeta_{n,\eps}
 =a_n-\frac{a_n^2}{30}\eta^2
  +\left(\frac{a_n^3}{525}-\frac1{560}\right)\eta^4
  +O_n(\eta^6).
 \label{eq:scaled-zero-refinement}
\end{equation}

For completeness, the local labels agree with the global positive-zero
indices. The difference equation
\[
 qzA_q(q^2z)-A_q(qz)+A_q(z)=0
\]
implies $A_q(z)>0$ on $0\le z\le1/(4q)$.
Indeed set $y_k=A_q(q^kz)$; then $y_k\to1$ and
$y_k=y_{k+1}-\alpha_k y_{k+2}$, where
$0\le\alpha_k=q^{k+1}z\le1/4$.
For large $k$, positivity and the recurrence give
$1/2\le y_k/y_{k+1}\le1$. Backward induction preserves this interval,
since $1-\alpha_k/(y_{k+1}/y_{k+2})\in[1/2,1]$.

Choose $b_n\in(a_{n+1},a_n)$ and small disjoint Airy-zero disks whose
closed real diameters lie in $(b_n,0)$. Uniform compact convergence of
$\mathcal U_\eps$ gives one real simple zero in each disk and no
other real zero on the complement in $[b_n,0]$. Since
$X-\zeta=O(\eta^2)$ on these disks, the zeros are the ones
constructed above. The zero-free endpoint $1/(4q)$ corresponds to
$\zeta=-3\eta/2$. Hence there are no additional earlier positive zeros.
The map $z=\frac14\exp(-\eps/2-\eta^2\zeta)$ reverses the real
ordering, so these are precisely the first $n$ positive zeros.
Applying this map to \eqref{eq:scaled-zero-refinement} proves
\eqref{eq:zero-exp-form}.
\end{proof}

\subsection{The normalized difference equation}
\label{sec:difference-airy}

The turning-point scaling also reduces the $q$-difference equation for
$A_q$ to the Airy equation. We show that the next nonvanishing term gives
the equation satisfied by the first correction in
Corollary~\ref{prop:main-expansion}.

Coefficient comparison in \eqref{eq:Aq-def} gives
\begin{equation}
 qzA_q(q^2z)-A_q(qz)+A_q(z)=0.
 \label{eq:Ramanujan-equation}
\end{equation}
With $\eta=\eps^{1/3}$, the normalized function satisfies
\begin{equation}
 \Ue(\zeta)
 -2e^{\eta^2\zeta/2+\eta^3/4}\Ue(\zeta+\eta)
 +e^{-\eta^3/2}\Ue(\zeta+2\eta)=0.
 \label{eq:normalized-difference}
\end{equation}
Indeed multiplication of $z_\eps(\zeta)$ by $q$ shifts $\zeta$
to $\zeta+\eta$. The exponential factor multiplying $\Ue$ when
\eqref{eq:normalization} is solved for $A_q$ changes by the factors
\[
 \exp\left(\frac{\log4}{2}+\frac{\eta^2\zeta}{2}
                  +\frac{\eta^3}{4}\right),\qquad
 \exp\left(\log4+\eta^2\zeta+\eta^3\right)
\]
under shifts by $\eta$ and $2\eta$, respectively. Substitution in
\eqref{eq:Ramanujan-equation} gives \eqref{eq:normalized-difference}.

For a fixed function $f$ holomorphic near a compact set $K$, the
left side of \eqref{eq:normalized-difference}, with $\Ue$ replaced
by $f$, has the expansion
\begin{align}
 {}&\eta^2(f''-\zeta f)+\eta^3(f'''-\zeta f'-f)\notag\\
 &\quad+\eta^4\left(\frac7{12}f^{(4)}-\frac\zeta2f''
              -\frac32f'-\frac{\zeta^2}{4}f\right)
 +O_{K,f}(\eta^5).
 \label{eq:difference-expansion}
\end{align}
Taylor's formula and Cauchy estimates on a larger compact set justify
the uniform remainder. The leading equation is $f''=\zeta f$.
The coefficient of $\eta^3$ is its derivative and vanishes on every
Airy solution. For a formal expansion
$f_\eta=f_0+\eta^2g+O(\eta^3)$ with $f_0''=\zeta f_0$,
the next coefficient requires
\[
 g''-\zeta g=\frac13f_0'+\frac{\zeta^2}{6}f_0.
\]
A particular solution is
$g=(4\zeta f_0+\zeta^2f_0')/30$, in agreement with
\eqref{eq:main-expansion}. The difference equation permits additional
homogeneous Airy terms; the integral analysis determines their coefficients.

\section{The Stieltjes--Wigert Edge Kernel}
\label{sec:kernel}

We apply Corollary~\ref{prop:main-expansion} to the edge kernel of the
Stieltjes--Wigert ensemble. Its limit as $q\to1^-$ is the Airy kernel;
see Takahashi \cite[Section~3.2 and Appendix~C]{Takahashi2014} and
Forrester \cite[Proposition~1.3]{Forrester2022}. We calculate the first
correction in the logarithmic coordinates used below. The normalization
is specified through the finite orthonormal sum, since both the
coordinates and the reference measure affect this correction.

For $q=e^{-\eps}$, define
\begin{equation}
 w_\eps(u)=\frac1{\sqrt{2\pi\eps}}
       \exp\left(-\frac{(\log u)^2}{2\eps}\right),\qquad u>0.
 \label{eq:sw-weight}
\end{equation}
The ensemble has density proportional to
$\prod_{j=1}^Nw_\eps(u_j)\prod_{j<k}(u_k-u_j)^2$ with respect to
$\dd u_1\cdots\dd u_N$. Its orthonormal polynomials are
\begin{equation}
 p_n(u)=(-1)^nq^{n/2+1/4}\sqrt{(q;q)_n}
 \sum_{k=0}^n\frac{q^{k^2}(-\sqrt q\,u)^k}
 {(q;q)_k(q;q)_{n-k}}.
 \label{eq:sw-polynomials}
\end{equation}
These conventions agree with (1.18)--(1.21) in the preprint
version of \cite{Forrester2022}.
They can also be checked from the moments
\[
 \int_0^\infty u^j w_\eps(u)\,\dd u=q^{-(j+1)^2/2}
\]
and the finite $q$-binomial theorem. The correlation kernel is
\[
 K_{N,\eps}^{\rm SW}(u,v)
 =\sqrt{w_\eps(u)w_\eps(v)}\sum_{n=0}^{N-1}p_n(u)p_n(v).
\]

For fixed $q$, its limit as $N\to\infty$ is
\begin{align}
 K_\eps^{\rm SW}(u,v)
 ={}&\frac{\sqrt{w_\eps(u)w_\eps(v)}}{(q;q)_\infty}
 \frac{A_q(\sqrt q\,u)A_q(q^{-1/2}v)
       -A_q(\sqrt q\,v)A_q(q^{-1/2}u)}{u-v},
 \label{eq:sw-kernel}
\end{align}
with the removable value on the diagonal. We record the normalization
calculation. Uniformly on compact sets of complex $u$, the defining sum
gives $p_n(u)=O(q^{n/2})$ and
\begin{equation}
 p_n(u)=\frac{(-1)^nq^{n/2+1/4}\sqrt{(q;q)_n}}{(q;q)_\infty}
 \left\{A_q(\sqrt q\,u)-\frac{q^{n+1}}{1-q}A_q(q^{-1/2}u)
                  +O(q^{2n})\right\}.
 \label{eq:sw-polynomial-limit}
\end{equation}
Indeed, replace $1/(q;q)_{n-k}$ by
$(q^{n-k+1};q)_\infty/(q;q)_\infty$ and use
$(a;q)_\infty=1-a/(1-q)+O(a^2)$ for $0\le a\le q$.
For fixed $R>0$, the resulting error on $|u|\le R$ is bounded by
a constant times $q^{2n}\sum_{k\ge0}q^{k^2-2k}R^k/(q;q)_k$;
the omitted tails of the two entire series satisfy the same bound.
The leading coefficient of $p_n$ is
$q^{n^2+n+1/4}/\sqrt{(q;q)_n}$. Substitution in the
Christoffel--Darboux formula gives \eqref{eq:sw-kernel}.
The bound $p_n=O(q^{n/2})$ also proves local uniform convergence of
the orthonormal sum, with error $O(q^N)$, including its diagonal.
The constants here may depend on $q$.

Put
\begin{equation}
 u_\eps(s)=\frac14e^{-\eps^{2/3}s},\qquad
 \widehat K_\eps(s,t)=\eps^{2/3}\sqrt{u_\eps(s)u_\eps(t)}
 K_\eps^{\rm SW}(u_\eps(s),u_\eps(t)).
 \label{eq:sw-scaling}
\end{equation}
For real $s,t$, this is the transformed correlation kernel with respect
to Lebesgue measure. The factor is
$\sqrt{|u_\eps'(s)u_\eps'(t)|}$, since $u_\eps$ is decreasing.
We use its analytic continuation for complex $s,t$.
Let
\begin{equation}
 K_{\rm Ai}(s,t)=\frac{\Ai(s)\Ai'(t)-\Ai'(s)\Ai(t)}{s-t},
 \qquad K_{\rm Ai}(s,s)=\Ai'(s)^2-s\Ai(s)^2.
 \label{eq:airy-kernel}
\end{equation}

\begin{proposition}
\label{prop:kernel}
For every compact $K\subset\C$, as $\eps\downarrow0$, uniformly for
$(s,t)\in K^2$,
\begin{equation}
 \widehat K_\eps(s,t)=K_{\rm Ai}(s,t)
       +\eps^{2/3}K_2(s,t)+O_K(\eps^{4/3}),
 \label{eq:kernel-expansion}
\end{equation}
where
\begin{align}
 30K_2(s,t)={}&(s+t)\Ai'(s)\Ai'(t)
       -(s^2+st+t^2)\Ai(s)\Ai(t)\notag\\
 &+4\bigl(\Ai(s)\Ai'(t)+\Ai'(s)\Ai(t)\bigr).
 \label{eq:kernel-correction}
\end{align}
The expansion holds through the diagonal and after any fixed number of
derivatives in either variable. In particular,
\begin{align}
 \widehat K_\eps(s,s)
 ={}&\Ai'(s)^2-s\Ai(s)^2\notag\\
 &+\frac{\eps^{2/3}}{30}
  \bigl(2s\Ai'(s)^2-3s^2\Ai(s)^2+8\Ai(s)\Ai'(s)\bigr)
  +O_K(\eps^{4/3}).
 \label{eq:kernel-density}
\end{align}
\end{proposition}

\begin{proof}
Write $\eta=\eps^{1/3}$ and use $\Ue$ from
\eqref{eq:normalization}. Define
\begin{equation}
 V_\eps(s)=\frac{\Ue(s)-e^{-\eta^2s/2}\Ue(s-\eta)}\eta.
 \label{eq:kernel-difference}
\end{equation}
Since $q^{-1/2}u_\eps(s)=\sqrt q\,u_\eps(s-\eta)$,
substitution in \eqref{eq:sw-kernel} gives the exact identity
\begin{equation}
 \widehat K_\eps(s,t)
 =\frac{e^{5\eps/24}}{(\widetilde q;\widetilde q)_\infty}
  \frac{\eta^2(s-t)/2}{\sinh(\eta^2(s-t)/2)}
  \frac{\Ue(s)V_\eps(t)-\Ue(t)V_\eps(s)}{s-t},
 \label{eq:kernel-exact}
\end{equation}
where $\widetilde q=e^{-4\pi^2/\eps}$. Here we used
\[
 (q;q)_\infty=\sqrt{\frac{2\pi}{\eps}}
 e^{-\pi^2/(6\eps)+\eps/24}
 (\widetilde q;\widetilde q)_\infty.
\]
Replacing $s$ by $s-\eta$ multiplies the reciprocal normalizing
factor in \eqref{eq:normalization} by
$\frac12e^{-\eta^2s/2+\eps/4}$. Together with the weight and
\[
 \frac{u_\eps(s)-u_\eps(t)}{\sqrt{u_\eps(s)u_\eps(t)}}
 =-2\sinh\left(\frac{\eta^2(s-t)}2\right),
\]
this gives \eqref{eq:kernel-exact}.

In this proof write $A=\Ai$ and
$G(s)=(4sA(s)+s^2A'(s))/30$.
Corollary~\ref{prop:main-expansion} gives
$\Ue=A+\eta^2G+\eta^3A/48+O_K(\eta^4)$.
Taylor expansion of \eqref{eq:kernel-difference} yields
\begin{align*}
 V_\eps={}&\Ue'-\frac\eta2(\Ue''-s\Ue)
 +\eta^2\left(\frac{\Ue'''}6-\frac{s\Ue'}2\right)\\
 &+\eta^3\left(-\frac{\Ue^{(4)}}{24}
          +\frac{s\Ue''}4-\frac{s^2\Ue}8\right)+O_K(\eta^4).
\end{align*}
Apply Cauchy estimates on a fixed larger disk containing the convex
hull of $K$ and all the shifted arguments. In particular, a holomorphic
remainder $r_\eps=O_K(\eta^4)$ contributes
\[
 \eta^{-1}\bigl(r_\eps(s)-e^{-\eta^2s/2}r_\eps(s-\eta)\bigr)
 =O_K(\eta^4),
\]
because $r_\eps'=O_K(\eta^4)$ and the exponential increment is
$O_K(\eta^2)$. Using $A''=sA$ and
$G''-sG=A'/3+s^2A/6$, we obtain
\begin{equation}
 V_\eps(s)=A'(s)
 +\frac{\eta^2}{30}\bigl((s^3+9)A(s)-4sA'(s)\bigr)
 -\frac{11\eta^3}{48}A'(s)+O_K(\eta^4).
 \label{eq:kernel-difference-expansion}
\end{equation}
Substitution gives
\[
 \frac{\Ue(s)V_\eps(t)-\Ue(t)V_\eps(s)}{s-t}
 =K_{\rm Ai}(s,t)+\eta^2K_2(s,t)
       -\frac{5\eta^3}{24}K_{\rm Ai}(s,t)+O_K(\eta^4).
\]
The error before division is holomorphic and antisymmetric. It
vanishes at $s=t$, so Cauchy estimates on a larger disk preserve its
order after division by $s-t$. Finally,
\[
 \frac{e^{5\eps/24}}{(\widetilde q;\widetilde q)_\infty}
 =1+\frac{5\eta^3}{24}+O(\eta^6)+O(e^{-4\pi^2/\eps}),
 \qquad
 \frac{\eta^2(s-t)/2}{\sinh(\eta^2(s-t)/2)}=1+O_K(\eta^4).
\]
The terms of order $\eta^3$ cancel. This proves
\eqref{eq:kernel-expansion}. Cauchy estimates on larger compact
sets also prove the assertions about derivatives and the diagonal.
\end{proof}

The two limits in this application are successive: first $N\to\infty$
with $q$ fixed, then $q\to1^-$. The result gives uniform expansions
of each fixed correlation determinant on bounded coordinate sets.
A joint limit in $N$ and $q$, or a limit for an extreme-eigenvalue
distribution, requires additional estimates uniform in $q$ and on
unbounded sets. Polynomial combinations of Airy functions also occur
in finite-matrix corrections for classical ensembles; see
Choup \cite[Theorem~1.2]{Choup2006}. Here the correction concerns the
Stieltjes--Wigert kernel in the precise scaling
\eqref{eq:sw-scaling}.

\section{Numerical Illustrations}
\label{sec:numerics}

The defining series is evaluated using
\[
 t_0=1,\qquad \frac{t_{k+1}}{t_k}
 =\frac{-zq^{2k+1}}{1-q^{k+1}}.
\]
The absolute ratios $r_k=|t_{k+1}/t_k|$ decrease with $k$. Once $r_k$ is
less than one, the remaining absolute tail is at most
$|t_k|r_k/(1-r_k)$. We use 220 decimal digits and stop when this bound
is smaller than $10^{-190}\max\{1,|S_k|\}$, where $S_k$ is the
partial sum. Selected calculations repeated at 180 digits agreed to
more than 100 decimal places. These computations illustrate the
asymptotic formulas.

Table~\ref{tab:uniform-error} gives the absolute error and the scaled
error in \eqref{eq:fixed-expansion} on the set
\[
 \mathcal S=\{0.23,0.25,0.27,0.24\pm0.01i,0.26\pm0.01i\}.
\]
Here $R_\eps$ is evaluated by subtracting the two displayed Airy terms
from the normalized series, and $E$ is evaluated at
$\chi(x)/\eps^{2/3}$. The calculations do not determine an admissible
radius $\delta$ in Theorem~\ref{thm:fixed}.

\begin{table}[htbp]
\centering
\caption{Errors in the uniform expansion on $\mathcal S$.}
\label{tab:uniform-error}
\footnotesize
\begin{tabular}{@{}ccc@{}}
\hline
$\eps$ & $\max_{x\in\mathcal S}|R_\eps(x)|$
& $\max_{x\in\mathcal S}|R_\eps(x)|/(\eps^2E)$\\
\hline
0.10000 & \(5.513681\times10^{-6}\) & \(5.950193\times10^{-4}\)\\
0.05000 & \(1.519744\times10^{-6}\) & \(6.799869\times10^{-4}\)\\
0.02500 & \(4.146847\times10^{-7}\) & \(7.787581\times10^{-4}\)\\
0.01250 & \(1.215561\times10^{-7}\) & \(7.589149\times10^{-4}\)\\
0.00625 & \(4.778404\times10^{-8}\) & \(4.388401\times10^{-4}\)\\
\hline
\end{tabular}
\end{table}

At $x=1/4$, the successive order of the absolute error for the last
two values of $\eps$ is $1.99858$, consistent with the factor
$\eps^2$ in \eqref{eq:fixed-error}.

For the zeros, Newton iteration is applied to the directly summed
$\Ue$, starting at $a_n-a_n^2\eps^{2/3}/30$. We apply
$z_\eps$ to the computed root and let $\Delta_n(\eps)$ be its
logarithm, plus $\log4$, minus the four explicit terms on the
right side of \eqref{eq:zero-exp-form}.
Table~\ref{tab:zero-error} gives the scaled errors for the branches
associated with $a_1,a_2,a_3$. Corollary~\ref{cor:zeros} identifies
their global indices for all sufficiently small $\eps$.
The programs and complete numerical results are supplied with the paper.

\begin{table}[htbp]
\centering
\caption{Scaled logarithmic zero errors $\Delta_n(\eps)/\eps^{8/3}$.}
\label{tab:zero-error}
\footnotesize
\begin{tabular}{@{}crrr@{}}
\hline
$\eps$ & $n=1$ & $n=2$ & $n=3$\\
\hline
0.10000 & -0.002197659 & 0.004817044 & 0.025816005\\
0.05000 & -0.002106308 & 0.005639316 & 0.029285791\\
0.02500 & -0.002045995 & 0.006154187 & 0.031383772\\
0.01250 & -0.002006971 & 0.006476461 & 0.032666955\\
0.00625 & -0.001981999 & 0.006678441 & 0.033459183\\
\hline
\end{tabular}
\end{table}

For the kernel, Table~\ref{tab:kernel-error} gives
\[
 \frac{\widehat K_\eps(s,t)-K_{\rm Ai}(s,t)
                  -\eps^{2/3}K_2(s,t)}{\eps^{4/3}}
\]
at one off-diagonal and two diagonal pairs. We evaluate $A_q$ and its
first derivative by their defining series and use the removable
quotient on the diagonal. The corrected-error orders between the last
two values of $\eps$ are $1.33738$, $1.34588$, and $1.32743$,
respectively, consistent with Proposition~\ref{prop:kernel}.
Independent finite orthonormal sums at $q=1/2$ also check the sign and
normalization in \eqref{eq:sw-kernel}; the data and program are included
with the numerical supplement.

\begin{table}[htbp]
\centering
\caption{Kernel errors divided by $\eps^{4/3}$ after the first correction.}
\label{tab:kernel-error}
\footnotesize
\begin{tabular}{@{}crrr@{}}
\hline
$\eps$ & $(s,t)=(-2,-1)$ & $(s,t)=(0,0)$ & $(s,t)=(2,2)$\\
\hline
0.10000 & -0.008610739 & 0.000347051 & 0.000339311\\
0.05000 & -0.008505378 & 0.000334537 & 0.000345182\\
0.02500 & -0.008442346 & 0.000327079 & 0.000348846\\
0.01250 & -0.008404067 & 0.000322545 & 0.000351140\\
0.00625 & -0.008380541 & 0.000319753 & 0.000352579\\
\hline
\end{tabular}
\end{table}

\section*{Acknowledgments}
The author thanks Professor M.~E.~H. Ismail for drawing attention to Hayman's
work on zeros and for suggesting the connection with random matrix theory.

The work was supported by the Department of Education of Guangdong Province
[Grant No.~2025ZDJ099] and Nanfang College, Guangzhou
[Grant No.~2026BQ010].

\section*{Declaration of AI Use}
ChatGPT (OpenAI) assisted with mathematical derivations, literature
searches and comparison of published results, symbolic and numerical
checks, preparation of checking programs, and drafting and revising
the manuscript. The author is responsible for the mathematical arguments,
references, and final text.

\end{document}